\documentclass[12pt]{elsarticle}
\usepackage{amssymb}
\usepackage{array}
\usepackage{latexsym}
\usepackage{enumerate}
\usepackage{amsmath}
\usepackage{amsfonts}
\usepackage{amsthm}
\usepackage{graphicx}
\usepackage{comment}
\usepackage[english]{babel}
\usepackage{microtype}
\usepackage{geometry}
\usepackage{doi}
\usepackage{hyperref}
\ifx\smallsetminus\undefined
\def\smallsetminus{\setminus}
\fi
\date{}
\theoremstyle{plain}
\newtheorem{prop}{Proposition}[section]
\newtheorem{theorem}[prop]{Theorem}

\newtheorem{lemma}[prop]{Lemma}
\theoremstyle{definition}

\newcommand{\cR}{{\mathcal R}}
\newcommand{\cC}{\mathcal C}

\newcommand{\cH}{\mathcal H}
\newcommand{\cQ}{\mathcal Q}

\newcommand{\fA}{\mathfrak{A}}
\newcommand{\fB}{\mathfrak B}
\newcommand{\fC}{\mathfrak{C}}
\newcommand{\cG}{\mathcal{G}}
\newcommand{\PG}{\mathrm{PG}}
\newcommand{\PGU}{\mathrm{PGU}}
\newcommand{\AG}{\mathrm{AG}}
\newcommand{\GF}{\mathrm{GF}}

\newcommand{\rank}{\operatorname{rank}}
\newcommand{\fx}{\mathfrak x}
\newcommand{\fy}{\mathfrak y}
\begin{document}
\begin{frontmatter}
\title{Intersections of the Hermitian surface with irreducible quadrics in
  $\PG(3,q^2)$, $q$ odd}
\author[AA]{Angela Aguglia}
\ead{angela.aguglia@poliba.it}
\address[AA]{Dipartimento di Meccanica, Matematica e Management, Politecnico di
Bari, Via Orabona 4, I-70126 Bari, Italy}
\author[LG]{Luca Giuzzi\corref{cor1}}
\ead{luca.giuzzi@unibs.it}
\address[LG]{DICATAM - Section of Mathematics,
University of Brescia,
Via Branze 43, I-25123, Brescia, Italy}
\cortext[cor1]{Corresponding author. Tel. +39 030 3715739; Fax. +39 030 3615745}
\begin{abstract}
In  $\PG(3,q^2)$, with $q$ odd,
we determine  the possible intersection
sizes of a Hermitian surface $\cH$ and an irreducible  quadric $\cQ$
having the same tangent plane at a common point $P\in\cQ\cap\cH$.
\end{abstract}
\end{frontmatter}
\section{Introduction}
The study of the intersection of  hypersurfaces in projective spaces
is a deep classical  problem in algebraic geometry; see \cite{F} for the
general theory.
Many of the usual properties over algebraically closed
fields, e.g. B\'ezout's theorem, fail to fully hold in the finite case.
There are several results on the combinatorial characterization of
the intersection pattern of curves and surfaces
as there is a close relationship between the size of the
intersection of two varieties and the weight distribution of some
functional linear codes; see \cite{L}.
The problem, when both surfaces are of the same degree has been
widely investigated; see for instance
\cite{BH,EHRS10,CG} for the case of quadrics
or \cite{G1,DD1,DD2} for that of Hermitian surfaces.
\par
When one of the variety is Hermitian and the other is a
quadric the problem appears to be more difficult to tackle.
For instance, the possible intersection patterns between Hermitian curves and
conics have been studied in \cite{DDK},
whereas the possible intersections  between a conic and a curve of a
non-classical Buekenhout-Metz unital in $\PG(2,q^2)$
have been determined in \cite{AG}.

In a recent series of papers, the largest values for
the spectrum of the intersection of Hermitian and quadratic varieties
have been investigated, as it is related to the minimum
distance of certain functional codes. We refer to
\cite{E0,E,EHRS,HS,ELX} for codes defined on a Hermitian variety by quadratic
forms and \cite{BB} for the converse, i.e.  codes defined by Hermitian
forms on a quadric. In both cases, it appears that the maximum cardinality
should be attained when one of the varieties splits in the union of hyperplanes,
even if this, in the case of \cite{BB}, is, for the time being, still an open
conjecture.

The setting of the present paper is slightly different. Here we restrict
ourselves to dimension $3$ and assume the characteristic to be odd.
We aim to provide the spectrum of all possible intersection numbers
between a Hermitian surface
and an irreducible quadric, under the further assumption that they
share a common tangent plane.
Our main result is contained in the following Theorem \ref{main1}.

\begin{theorem}\label{main1}
In $\PG(3,q^2)$, with $q$ odd, let $\cH$ and $\cQ$
be respectively a Hermitian surface and
an irreducible quadric  having the same tangent plane $\pi$ at a common 
nonsingular point $P$.
Then, the possible intersection sizes for $\cH\cap\cQ$ in $\PG(3,q^2)$ are
as follows.
\begin{itemize}
\item For $\cQ$ elliptic:
\[q^3-q^2+1,q^3-q^2+q+1, q^3-q+1, q^3+1, q^3+q+1, q^3+q^2-q+1, q^3 + q^2+1. \]
\item For $\cQ$ a $1$--degenerate cone and $q>3$,
\[ q^3-q^2+1,q^3-q^2+q+1, q^3-q+1, q^3+1, q^3+q+1,\]
\[ q^3+q^2-q+1, q^3 + q^2+1,  q^3 + 2q^2-q+1. \]
\item For $\cQ$ hyperbolic:
\[ q^2+1,q^3-q^2+1,  q^3-q^2+q+1, q^3-q+1, q^3+1,  q^3+q+1,
q^3+q^2-q+1, q^3+ q^2+1, \]
\[   q^3+2q^2-q+1,  q^3+2q^2+1,  q^3+3q^2-q+1,   2q^3+q^2+1.  \]
\end{itemize}
\end{theorem}
\noindent
The proof of this theorem is contained in Section \ref{PT}.

In Section \ref{EC} we characterize
the geometric configurations corresponding respectively
to the minimum and maximum cardinality allowable: in both cases,
detailed in Theorems \ref{ppt1} and \ref{ppt2}, the quadric
$\cQ$ is hyperbolic and in permutable position with
the Hermitian surface $\cH$.

\section{Proof of Theorem \ref{main1}}
\label{PT}

Fix a projective frame in $\PG(3,q^2)$ with homogeneous coordinates
$(J,X,Y,Z)$, and consider the affine space $\AG(3,q^2)$
whose infinite hyperplane $\Sigma_{\infty}$ has equation $J=0$.
Then, $\AG(3,q^2)$ has affine coordinates $(x,y,z)$,
where $x=X/J$, $y=Y/J$ and $z=Z/J$.

Since all non-degenerate Hermitian surfaces are projectively equivalent,
we can assume without loss of generality $\cH$ to have affine equation
\begin{equation}\label{eq}
z^q+z=x^{q+1}+y^{q+1}.
\end{equation}
The unitary group $\PGU(4,q)$ is transitive on the points of $\cH$;
see \cite{S}.
Thus, we can also suppose $P\in\cH\cap\cQ$ to be
$P=P_{\infty}(0,0,0,1)$;
the common tangent plane $\pi$ between $\cQ$ and
$\cH$ at $P$ is then $\Sigma_{\infty}$.
Under the aforementioned assumptions,
the equation of the irreducible quadric $\cQ$  is of the form
\begin{equation}\label{eq2}
z=ax^{2}+by^{2}+cxy+dx+ey+f,
\end{equation}
where $(a,b,c)\neq (0,0,0)$ and $\Delta=4ab-c^2\neq 0$
when $\cQ$ is non-singular.

Write $\cC_{\infty}:=\cQ\cap\cH\cap\Sigma_{\infty}$.
When $\cQ$ is elliptic, that is $\Delta$ is a nonsquare in $\GF(q^2)$,
the point $P_{\infty}$ is, clearly, its only point at infinity of
the intersection; that is
$\cC_{\infty}=\{P_{\infty}\}$.
The structure of
$\cC_{\infty}$ when $\cQ$ is hyperbolic or a $1$--degenerate cone, is
detailed by the following lemma.

\begin{lemma} If $\cQ$ is a $1$--degenerate cone, then $\cC_{\infty}$
  consists of either one or $q^2+1$ points on a line.
  When $\cQ$ is a hyperbolic quadric, then $\cC_{\infty}$ consists of
  either one, or $q^2+1$ or $2q^2+1$ points. All cases may
  actually occur.
\end{lemma}
\begin{proof}
As both $\cH\cap\Sigma_{\infty}$ and $\cQ\cap\Sigma_{\infty}$ split in
lines through $P_{\infty}$, it is straightforward to
see that the only possibilities for $\cC_{\infty}$ are the following:
when $\cQ$ is a $1$--degenerate cone  $\cC_\infty$ is either a point
or one  line,  whereas, when $\cQ$ is a hyperbolic quadric, $\cC_{\infty}$
consists of either a point or one or two lines.
Actually, the set $\cC_{\infty}$ is determined by the following equations
\begin{equation}\label{sinf}
\begin{cases}
  x^{q+1}+y^{q+1}=0 \\
  ax^2+by^2+cxy=0.
\end{cases}
\end{equation}
Suppose $\cQ$ to be a $1$--degenerate cone.
Under our assumptions, when $c=0$ we can assume either $a=0$ or $b=0$.
In both cases $\cC_{\infty}$ is the point $P_{\infty}$.

If $c\neq 0$, then the vertex of $\cQ$ is the point $V=(0,-c,2a,e-dc)$.
As $P_{\infty}$ is not the vertex of $\cQ$, the conic
$\cC_{\infty}$ consists of a line $\ell$ if and only if
$V\in\cH$.
This is the same as to require $c^{q+1}+4a^{q+1}=0 $, that is $||c||=-4||a||$, where $||x||$ denotes the norm of $x\in\GF(q^2)$ over $\GF(q)$.

Suppose now that $\cQ$ is a hyperbolic quadric.
We consider the intersection of $\cC_{\infty}$
with a line $\ell$ of $\Sigma_{\infty}$ not through $P_{\infty}$.
When $a=0$, let $\ell:x=1$; then from \eqref{sinf} we get
\begin{equation}\label{sinf3b}
\begin{cases}
  1+y^{q+1}=0 \\
  y(by+c)=0.
\end{cases}
\end{equation}
This system admits solution if and only if $b\neq 0$ and
$||cb^{-1}||=-1$. In this case, clearly, its solution is unique
and $\cC_{\infty}$ is just a line.
An analogous argument applies when $b=0$ and $a\neq 0$.

Suppose now $a,b\neq 0$. We can take $\ell: y=1$.
Then, from System \eqref{sinf} we obtain
\begin{equation}\label{sinf2}
\begin{cases}
  x^{q+1}+1=0 \\
  x^2+\frac{c}{a}x+\frac{b}{a}=0.
\end{cases}
\end{equation}
Set $s^2=c^2-4ab\neq 0$.
The solutions of the second equation in  \eqref{sinf2} are
\[x_{1/2}=\frac{-c\pm s}{2a}.\]
If $c=0$, then \eqref{sinf2} has either
$2$ or $0$ solutions, according as $||s||=-4||a||$ or not.
Consequently,
$\cC_{\infty}$ consists of either $2q^2+1$ points or
just one point, namely $P_{\infty}$.

If $c\neq 0$ and $||s-c||=||s+c||=-4||b||$, then
$x_1^{q+1}=x_2^{q+1}=-1$;  thus, $\cC_{\infty}$
contains  $2q^2+1$ points.

If  $c\neq 0$, $||s-c||\neq ||s+c||$ and either
$||s+c||=-4||b||$ or  $||s-c||=-4||b||$, then
$\cC_{\infty}$ consists of $q^2+1$ points.
In all of the remaining cases, $\cC_{\infty}=\{P_{\infty}\}$.
\end{proof}

\begin{lemma}\label{main2}
Under the assumptions of Theorem \ref{main1},
the possible  sizes of $(\cH\cap\cQ)\setminus\pi$
are
either
\[q^3-q^2,q^3-q^2+q, q^3-q, q^3, q^3+ q, q^3+q^2-q,q^3+q^2\]
when $\cQ$ is elliptic or a $1$--degenerate cone
or
\[ q^2, q^3-q^2, q^3-q^2+q, q^3-q, q^3, q^3+ q, q^3+q^2-q,q^3+q^2, 2q^3-q^2\]
 when $\cQ$ is hyperbolic.
\end{lemma}
\begin{proof}

We use the same setup as at the beginning of this section.
Thus, we need to compute  
the
number of common points between $\cH$ and
$\cQ$ in $\AG(3,q^2)$; hence, 
we have to study the following system of equations
\begin{equation}
    \left\{\begin{array}{l}\label{sis1}
    z^q+z=x^{q+1}+y^{q+1}     \\
       z=ax^{2}+by^{2}+cxy+dx+ey+f.
      \end{array}\right.
  \end{equation}
Once the value of $z$ is recovered from the second equation and substituted
in the first, we obtain
\begin{equation}
\begin{array}{l}\label{eqc}
a^qx^{2q}+b^qy^{2q}+c^qx^qy^q+d^qx^q+e^qy^q+f^q
+ax^{2}+by^{2}\\+cxy+dx+ey+f
=x^{q+1}+y^{q+1}.
\end{array}
\end{equation}
We now need to determine the number of solutions of \eqref{eqc}
as $a,b,c,d,e,f$
vary in $\GF(q^2)$.
To this purpose, choose a primitive element $\beta$ of $\GF(q^2)$.
As $q$ is odd, $\varepsilon=\beta^{(q+1)/2}\in\GF(q^2)$ and
$\varepsilon^q=-\varepsilon$; furthermore, $\varepsilon^2$ is a primitive
element of $\GF(q)$.
Since $\varepsilon\not\in\GF(q)$,
it is immediate to see that $\{1,\varepsilon\}$ is a basis of
$\GF(q^2)$, regarded as a vector space over $\GF(q)$.
We shall consequently write
each $x\in\GF(q^2)$ as a $\GF(q)$--linear
combination $x=x_0+x_1\varepsilon$ with $x_0,x_1\in\GF(q)$.

By regarding $\GF(q^2)$ as
a $2$--dimensional vector space over
$\GF(q)$, \eqref{eqc} can be rewritten as
\begin{equation}
\begin{array}{l}
\label{eqodd1}
(2a_0-1)x_0^2+(2a_0+1)\varepsilon^2x_1^2+4\varepsilon^2a_1x_0x_1+(2b_0-1)y_0^2+\varepsilon^2(2b_0+1)y_1^2+\\
4\varepsilon^2b_1y_0y_1+2c_0x_0y_0+2\varepsilon^2c_0x_1y_1+2\varepsilon^2c_1x_0y_1+2\varepsilon^2c_1x_1y_0+\\
2d_0x_0+2\varepsilon^2d_1x_1+2e_0y_0+2\varepsilon^2e_1y_1+2f_0=0.
\end{array}
\end{equation}
It is thus possible to consider the solutions $(x_0,x_1,y_0,y_1)$ of
\eqref{eqodd1} as the affine points of the (possibly degenerate)
 quadric hypersurface
$\Xi$ of $\PG(4,q)$ 
associated to the symmetric $5\times 5$ matrix
 \[A=\begin{pmatrix}
   (2a_0-1) &2\varepsilon^2a_1&c_0& \varepsilon^2c_1 & d_0\\
   2\varepsilon^2a_1 & (2a_0+1)\varepsilon^2  & \varepsilon^2c_1&\varepsilon^2c_0 &\varepsilon^2d_1 \\
   c_0 & \varepsilon^2c_1& (2b_0-1) &2\varepsilon^2b_1 & e_0 \\
   \varepsilon^2c_1&\varepsilon^2c_0&2\varepsilon^2b_1 & (2b_0+1)\varepsilon^2 & \varepsilon^2 e_1 \\
   d_0 &\varepsilon^2d_1 & e_0 & \varepsilon^2 e_1 & 2f_0\\
 \end{pmatrix}. \]
The above argument shows that the number of affine points of $\Xi$ equals
the number of points in
$\AG(3,q^2)$ which lie in $\cH\cap\cQ$; using the results of
\cite{HT} it is possible to actually count these points.
To this purpose, first  determine the number  points
at infinity of $\Xi$.
These points are those of the quadric $\Xi_\infty$ of
$\PG(3,q)$  associated to the symmetric $4\times 4$
block matrix
\[ A_{\infty}=\begin{pmatrix} \fA & \fC \\
                              \fC & \fB
                            \end{pmatrix}, \]
where
\[ \fA:=\begin{pmatrix}
  2a_0-1 & 2\varepsilon^2a_1 \\
  2\varepsilon^2a_1 & (2a_0+1)\varepsilon^2 \\
  \end{pmatrix},\quad
 \fB:=\begin{pmatrix}
     2b_0-1 &2\varepsilon^2b_1\\
     2\varepsilon^2b_1 & (2b_0+1)\varepsilon^2\\
  \end{pmatrix}, \]
  \[
 \fC:=\begin{pmatrix}
    c_0& \varepsilon^2c_1\\
    \varepsilon^2c_1&\varepsilon^2c_0 \\
  \end{pmatrix}.
\]
Denote by $(t,u,v,w)$ homogeneous coordinates
for this $\PG(3,q)$.
Observe that
 \begin{equation}\label{main3e}
      \det A_{\infty}=(c^2-4ab)^{q+1}-4a^{q+1}-4b^{q+1}-2c^{q+1}+1.
\end{equation}
We first show
 that $\rank A_{\infty}\geq 2$.
Actually, if it were $\rank A_{\infty}=1$, then
\begin{equation}\label{rank1}
  \frac{1}{\varepsilon^2}\det \fC=
  c_0^2-\varepsilon^2c_1^2=0.
 \end{equation}
As $\varepsilon^2$ is a non-square in $\GF(q)$,
Condition \eqref{rank1} is equivalent to $c_0=c_1=0$.
Thus, $A_{\infty}$ would be of the form
\[A_{\infty}=\begin{pmatrix}
    2a_0-1 &2\varepsilon^2a_1&0& 0\\
2\varepsilon^2a_1 & (2a_0+1)\varepsilon^2  & 0&0 \\
    0 & 0& 2b_0-1 &2\varepsilon^2b_1\\
     0&0&2\varepsilon^2b_1 & (2b_0+1)\varepsilon^2\\
      \end{pmatrix}. \]
Denote by $R_1,R_2,R_3,R_4$ the rows of $A_{\infty}$.
For $\rank A_{\infty}=1$ and $R_1$ not  null,
there should be $\alpha_i$ such that
$R_i=\alpha_i R_1$ for $2\leq i\leq 4$.
As the last two entries of $R_1$ are $0$,
$\alpha_3=\alpha_4=0$ and the matrix $\fB$
would be null.
This gives $b_0= 2^{-1}=-2^{-1}$, a contradiction as the
characteristic of $\GF(q)$ is odd.
The case in which $R_1$ is null, but a different row is not,
is entirely analogous.
\par
Our next step is to prove
that when  $\rank A_{\infty}=2$, then $\cQ$ is hyperbolic.
Recall that a quadric $\cQ$ of the form
\eqref{eq2} is hyperbolic if and only if $4ab-c^2$ is a non-zero square
in $\GF(q^2)$.
We distinguish two cases.
\begin{itemize}
\item
Suppose $c=0$; since $q$ is odd,  neither $\fA$ nor $\fB$ can be null matrices. Furthermore, $\rank A_{\infty}=2$ implies $(a,b)\neq(0,0)$ and
$\det \fA=\det \fB=0$.
Hence, $a^{q+1}=b^{q+1}=2^{-2}$ and
$(ab^{-1})^{q+1}=1$. In particular, $ab^{-1}\in\langle\varepsilon^{q-1}\rangle$
is a square in $\GF(q^2)$. Consequently, $4ab\neq 0$ is also a square
in $\GF(q^2)$. As $c^2-4ab\neq 0$, the quadric $\cQ$ is non-singular
and, thus,   hyperbolic.
\item
Assume $c\neq 0$; then, $\det \fC\neq 0$.
Denote by $A^{ij}$ the $3\times 3$ minor of $A_{\infty}$ obtained by deleting
its $i$--th row and $j$--th column.
Clearly, as $\rank A_{\infty}=2$, we have
$\det A^{ij}=0$ for any $1\leq i,j\leq 4$.
In particular, the following system of equations holds:
\begin{equation}\label{eqr2} \begin{cases}
 \frac{1}{4\varepsilon^4}\det A^{41}-\frac{1}{4\varepsilon^4}\det A^{32}=(a_0-b_0)c_1+(b_1-a_1)c_0=0 \\
 \frac{1}{4\varepsilon^4}\det A^{31}-\frac{1}{4\varepsilon^2}\det A^{42}=(a_1+b_1)\varepsilon^2c_1-(a_0+b_0)c_0=0.
\end{cases} \end{equation}
As $\varepsilon^2c_1^2-c_0^2=-||c||\neq 0$, System \eqref{eqr2} has
exactly one solution in $a_0,a_1$, namely
\begin{equation}
\label{a0a1}
 a_0=\frac{(-b_0c_1+2b_1c_0)c_1\varepsilon^2-b_0c_0^2}{c^{q+1}},\
 a_1=\frac{(b_1c_1\varepsilon^2-2b_0c_0)c_1+b_1c_0^2}{c^{q+1}}.
\end{equation}
Replacing these values in $A^{ij}$ we get
$\det A^{ij}=0$ if and only if $c^{q+1}+4b^{q+1}=1$.
Thus,
\[ 4ab-c^2=-\frac{c}{c^q}\left(c^{q+1}+4b^{q+1}\right)=\frac{-1}{c^{q-1}}. \]
Since
 $q$ is odd, $4ab-c^2\neq 0$ is a square of $\GF(q^2)$ and $\cQ$ is hyperbolic.
\end{itemize}
Now we determine the number $N$ of affine points of $\Xi$;  
$N=|\Xi|-|\Xi_{\infty}|$.
  Exactly one of the following cases
(C\ref{C1})--(C\ref{oh2}) happens.
\begin{enumerate}[(C1)]
\item\label{C1}
  $\det A\neq 0$,  $\det A_{\infty}\neq 0$  and $\det A_{\infty}$ is a square.

In this case, $\Xi$ is a parabolic quadric,  $\Xi_{\infty}$ is a hyperbolic quadric  and
\[N= (q+1)(q^2+1)-(q+1)^2=q^3-q\] 

 \item\label{C2} $\det A\neq 0$,  $\det A_{\infty}\neq 0$  and $\det A_{\infty}$
   is a nonsquare.

  Here the quadric $\Xi_{\infty}$ is  elliptic  and 
   \[N=  (q+1)(q^2+1)-(q^2+1)=q^3+q \] 

\item\label{C3}
  $\det A=0$,  $\det A_{\infty}\neq 0$  and $\det A_{\infty}$ is a square.

 We have that $\Xi$ is a cone  projecting a hyperbolic quadric of $\PG(3,q)$;
 thus 
 \[ N= q(q+1)^2+1-(q+1)^2=q^3+q^2-q. \]

\item\label{C4} $\det A= 0$,  $\det A_{\infty}\neq 0$  and $\det A_{\infty}$
  is a nonsquare.

In this case  $\Xi $ is a cone  projecting an elliptic
quadric of $\PG(3,q)$ and thus 
\[ N= q(q^2+1)+1-(q^2+1)=q^3-q^2+q \] 

\item\label{C5} $\rank A=4$, $\rank A_{\infty}=3$.

We get that $\Xi_{\infty}$ is a cone comprising the join of a point to a conic.
Thus, either
\[N=q(q+1)^2+1-[q(q+1)+1]=q^3+q^2\] or
\[N=q(q^2+1)+1-[q(q+1)+1]=q^3-q^2,\]
according as $\Xi$ is a cone  projecting
a hyperbolic  or an elliptic  quadric.

\item\label{C6} $\rank A=\rank A_{\infty}=3$.

Here $\Xi$  is the join of a line to a conic; so 
\[N= q^3+q^2+q+1-(q^2+q+1)=q^3\] 

\item\label{oh1}\label{C7} $\rank A=3$, $\rank A_{\infty}=2$.

We get that either \[N=q^3+q^2+q+1-(2q^2+q+1)=q^3-q^2\] or
\[N=q^3+q^2+q+1-q-1=q^3+q^2,\]
according as $\Xi_{\infty}$ is a pair of planes  or  a line.

\item\label{oh2}\label{C8} $\rank A=\rank A_{\infty}=2$.

 In this case either
  $\Xi$ is a pair of solids and $\Xi_{\infty}$ is a pair of planes
  or $\Xi$ is a plane and
  $\Xi_{\infty}$ is a line. Thus we get either
 \[N=2q^3+q^2+q+1-(2q^2+q+1)=2q^3-q^2\] or
\[N=q^2+q+1-(q+1)=q^2.\]
\end{enumerate}
Observe cases (C\ref{oh1}) and (C\ref{oh2}) are possible only when
$\cQ$ is hyperbolic.
\end{proof}
We now need
to provide compatibility conditions between
the set of the affine points of $\cH \cap \cQ$ and $\cC_{\infty}$, as
to explicitly determine $|\cH\cap\cQ|$.
In order to restrict the values of some of the parameters of \eqref{eq2}
we use a geometric argument.
\begin{lemma}\label{main4}
  If $\cQ$ is a hyperbolic quadric, we can assume without
  loss of generality:  \begin{enumerate}
  \item $b=0$, and $||c||\neq-||a||$
    if $\cC_{\infty}$ is just the point $P_\infty$;
  \item $b=0$, $c=\beta^{{q-1}/2}a$ if $\cC_{\infty}$ is a line;
  \item $b=-\beta^{(q-1)}a$, $c=0$ if $\cC_{\infty}$ is the union of
    two lines.
  \end{enumerate}
  When $\cQ$ is a cone, we can assume without loss of generality:
  \begin{enumerate}
    \item $b=c=0$ if $\cC_{\infty}$ is a point;
    \item $b=\beta^{q-1}a$, $c=2\beta^{(q-1)/2}a$ if $\cC_{\infty}$ is a line.
    \end{enumerate}
\end{lemma}
\begin{proof}
The stabilizer $\cG$ of $P_{\infty}$ in $\PGU(4,q)$ acts on the points of $\Sigma_{\infty}$
as the automorphism group of a degenerate Hermitian curve. As such, it
has three orbits on the points of $\Sigma_{\infty}$, namely, the
common tangency point $P_{\infty}$, the points of $\Sigma_{\infty}\cap\cH$ different
from $P_{\infty}$ and
those in $\Sigma_{\infty}\setminus\cH$; see \cite[\S 35 page 47]{S}.
Consequently, its
action on the lines through $P_{\infty}$ is the same as that of
$\PGU(2,q)$ on the points of $\PG(1,q^2)$: it affords
two orbits, say $\Lambda_1$ and
$\Lambda_2$ where $\Lambda_1$ consists of the
totally isotropic lines of $\cH$ through $P_{\infty}$ while $\Lambda_2$ contains
the remaining $q^2-q$ lines of $\Sigma_{\infty}$ through $P_{\infty}$.
Recall that $\cG$ is doubly transitive on
$\Lambda_1$ and the stabilizer of any $m\in\Lambda_1$ is
transitive on $\Lambda_2$. Let $\cQ_{\infty}=\cQ\cap\Sigma_{\infty}$.
If $\cQ$ is hyperbolic and $\cC_{\infty}=\{P_{\infty}\}$ we can assume
$\cQ_{\infty}$ to be the union of the line $\ell: x=0$ and
another line, say $ax+cy=0$ with $||a||\neq-||c||$. Thus, $b=0$.
Otherwise,
up to a suitable element $\sigma\in\cG$,
we can always take
$\cQ_{\infty}$ as the union of
any two lines in $\{\ell,m,n \}$
with
\[ \ell\!:\, x=0, \qquad m\!:\, x-\beta^{(q-1)/2}y=0,\qquad
n\!:\, x+\beta^{(q-1)/2}y=0.
\]
Actually,  when $\cC_{\infty}$ contains one line we take $\cQ_{\infty}=\ell\,m$,
while if $\cC_{\infty}$ is the union of two lines we have
$\cQ_{\infty}=mn$.
When $\cQ$ is a cone, we get either $\cQ_{\infty}=\ell^2$ or $\cQ_{\infty}=n^2$.
The lemma follows.
\end{proof}

\begin{lemma}
\label{l2.5}
  Suppose $\rank A_{\infty}=2$ and that $\cC_{\infty}$ is the union
  of two lines.
  Then, $|\Xi_{\infty}|=2q^2+q+1$.
\end{lemma}
\begin{proof}
  By Lemma \ref{main4},  assume $c=0$ and $b=-\beta^{q-1}a$.
 Observe that with this choice $b_0\neq \pm\frac{1}{2}$.
We distinguish two cases.
\begin{enumerate}
\item If $a=-\frac{1}{2}$, then $b=\frac{1}{2}\beta^{q-1}$. Let $M$ be
the nonsingular matrix
  \[ M=\begin{pmatrix}
    1 & 0 & 0 & 0 \\
    0 & 1 & 0 & 0 \\
    0 & 0 & (2b_0+1)\varepsilon^2 & 0 \\
    0 & 0 & -2b_1\varepsilon^2 & 1
    \end{pmatrix}; \]
a direct computation shows that
\[ M^TA_{\infty}M=\begin{pmatrix}
  -2 & 0 & 0 & 0 \\
   0 & 0 & 0 & 0 \\
   0 & 0 & 0 & 0 \\
   0 & 0 & 0 & (2b_0+1)\varepsilon^2
   \end{pmatrix}. \]
Here,
\[ 2b_0+1=b^q+b+1=\frac{1}{2}
\left(\frac{\beta^{q-1}+1}{\beta^{(q-1)/2}}\right)^2\!\!.\]

Hence $\Xi_{\infty}$ is projectively equivalent to
 \[-2t^2+(2b_0+1)\varepsilon^2w^2=0\]
and it is the union of two distinct planes if and only if
 $2(2b_0+1)\varepsilon^2$ is a square of $\GF(q)$.

As $(\beta^{(q+1)/2})^q=-\beta^{(q+1)/2}$,
 we have
$ \left(\frac{\beta^{q-1}+1}{\beta^{(q-1)/2}}\right)^q
=-\frac{\beta^{q-1}+1}{\beta^{(q-1)/2}}$.
Thus,
$\frac{\beta^{q-1}+1}{\beta^{(q-1)/2}}$ is not an element of $\GF(q)$.
It follows that
\[2(2b_0+1)\varepsilon^2=\left(\frac{\beta^{q-1}+1}{\beta^{(q-1)/2}}\right)^2\varepsilon^2\]
is the product of two non-squares and hence it is a square. 
The quadric $\Xi_{\infty}$ is reducible in the union of  two distinct planes and our lemma follows.
\item
 Consider now the case $2a_0+1\neq 0$.
 Take as $M$ the non-singular matrix
  \[ M=\begin{pmatrix}
    (2a_0+1)\varepsilon^2 & 0 & 0 & 0 \\
    -2a_1\varepsilon^2 & 1 & 0 & 0 \\
    0 & 0 & (2b_0+1)\varepsilon^2 & 0 \\
    0 & 0 & -2b_1\varepsilon^2 & 1
    \end{pmatrix}. \]
A straightforward computation proves that
\[ M^TA_{\infty}M=\begin{pmatrix}
   0 & 0 & 0 & 0 \\
   0 & (2a_0+1)\varepsilon^2 & 0 & 0 \\
   0 & 0 & 0 & 0 \\
   0 & 0 & 0 &  (2b_0+1)\varepsilon^2
 \end{pmatrix}.\]
 In particular, the quadric $\Xi_{\infty}$ defined by $A_{\infty}$ is
 projectively equivalent to
 \[(2a_0+1)u^2+(2b_0+1)w^2=0.\]
 From $4a^{q+1}=1= 4b^{q+1}$ and $b=-a\beta^{q-1}$ we get
 $a=\frac{\beta^{s(q-1)}}{2}$ and $b=-\frac{\beta^{(s+1)(q-1)}}{2}$
 for some value of $s\in\{0,1,\ldots,q+1\}$.
 As $2a_0+1=a^q+a+1$ and $2b_0+1=b^q+b+1$,
 \[2a_0+1=\frac{1}{2}\left(\frac{\beta^{s(q-1)}+1}{\beta^{s(q-1)/2}}\right)^2
 \!\!,\qquad
2b_0+1=-\frac{1}{2}
 \left(\frac{\beta^{(s+1)(q-1)}-1}{\beta^{(s+1)(q-1)/2}}\right)^2\!\!. \]
 If $s$ is odd, then
 $\left(\frac{\beta^{(s+1)(q-1)}-1}{\beta^{(s+1)(q-1)/2}}\right)$
and
 $\left(\frac{\beta^{s(q-1)}+1}{\beta^{s(q-1)/2}}\right)$ are in
 $\GF(q^2)\setminus\GF(q)$ while for $s$ even they are elements of $\GF(q)$.

Thus \[-(2a_0+1)(2b_0+1)=\frac{1}{4}
 \left(\frac{\beta^{(s+1)(q-1)}-1}{\beta^{(s+1)(q-1)/2}}\right)^2\left(\frac{\beta^{(s+1)(q-1)}-1}{\beta^{(s+1)(q-1)/2}}\right)^2\] is a square and
the quadric is reducible in the union of two  planes.

\end{enumerate}
\end{proof}
\begin{lemma}
\label{lc}
Suppose $\cQ$ to be a hyperbolic quadric with
$\rank A_{\infty}=2$  and that $\cC_{\infty}$ is just a point.
Then, $|\Xi_{\infty}|=q+1$.
\end{lemma}
\begin{proof}
By Lemma \ref{main4}, we may assume $b=0$.
The equations in \eqref{a0a1} give now  $a=0$ and $||c||=1$.
Take
\[ M=\begin{pmatrix}
  1 & 0 & c_0 & c_1\varepsilon^2 \\
  0 & 1 & -c_1&-c_0 \\
  0 & 0 & 1 & 0 \\
  0 & 0 & 0 & 1
  \end{pmatrix}.\]
Thus,
\[ M^TA_{\infty}M=\begin{pmatrix}
  -1 & 0 & 0 & 0 \\
  0 & \varepsilon^2 & 0 & 0 \\
  0 & 0 & c_0^2-c_1^2\varepsilon^2-1 & 0 \\
  0 & 0 & 0 & c_1^2\varepsilon^4+(1-c_0^2)\varepsilon^2
  \end{pmatrix} \]

As $||c||=c_0^2-\varepsilon^2c_1^2=1$,
the quadric $\Xi_{\infty}$ defined by
$A_{\infty}$ is equivalent to $\varepsilon^2u^2-t^2=0$. As
$\varepsilon\not\in\GF(q)$, this is the union of two conjugate
planes and it consists of just $q+1$ points.
\end{proof}

When $\cQ$ is elliptic, clearly $\cC_{\infty}=\{P_{\infty}\}$.
The possible sizes for the affine part of
$\cH\cap\cQ$
correspond to cases (C\ref{C1})--(C\ref{C6}) of Lemma \ref{main2}, whence
the theorem follows.

Consider now the case in which $\cQ$ is a cone. Here $\cC_{\infty}$ is
either a point or one line; once more,
the size of the affine part of $\cH\cap\cQ$
falls in one of cases (C\ref{C1})--(C\ref{C6}) of Lemma \ref{main2}.
If $\cC_{\infty}=\{P_{\infty}\}$, by Lemma \ref{main4}, we can assume
$b=c=0$. Under this assumption for $q>3$ all cases (C\ref{C1})--(C\ref{C6})
of Lemma \ref{main2} may occur. For $q=3$ the only square in $\GF(3)$ is 
$1$ and $\det A_{\infty}=1-4a^{q+1}=1$ yields $a=0$.
 Thus, cases (C\ref{C1}) and (C\ref{C3}) cannot occur.

When $\cC_{\infty}$ consists of one line, then, by Lemma \ref{main4},
we can assume
$b=\beta^{(q-1)}a$ and  $c=2\beta^{(q-1)/2}a$. Consequently,
$\det A_{\infty}=1$ and only cases (C\ref{C1}) and (C\ref{C3}) may occur.

If $\cQ$ is a hyperbolic quadric, then we have three possibilities
for $\cC_{\infty}$.
When $\cC_{\infty}$ consists of two lines, by Lemma \ref{main4}, we can assume
$b=-\beta^{(q-1)}a$, $c=0$. Consequently,
$\det A_{\infty}=(4a^{q+1}-1)^2$.  If $4a^{q+1}\neq 1$, then $\det
A_{\infty}$ is a non-zero square in $\GF(q)$ and cases (C\ref{C1}) and (C\ref{C3})
in
Lemma \ref{main2} may occur.
If $4a^{q+1}=1$, then $\rank A_{\infty}=2$ and either
(C\ref{C7}) or (C\ref{C8}) occurs. By Lemma \ref{l2.5}
it follows that in these latter cases
$\Xi_{\infty}$ is the union of two planes. Thus,
Case (C\ref{C7}) yields $|\cH\cap\cQ\cap\AG(3,q^2)|=q^3-q^2$, while from
Case (C\ref{C8}) we obtain $|\cH \cap \cQ \cap\AG(3,q^2)|=2q^3-q^2$.

Suppose now $\cC_{\infty}$ to be just one line.  By Lemma \ref{main4}
we can assume $b=0$, $c=\beta^{{q-1}/2}a$. This implies $\det
A_{\infty}=(a^{q+1}-1)^2$, which is a square in $\GF(q)$.
When this is zero we have
$\rank A_\infty=3$. Thus only cases (C\ref{C1}), (C\ref{C3}), (C\ref{C5})
and (C\ref{C6}) in Lemma \ref{main2} might happen.

Finally, if $\cC_{\infty}=\{P_{\infty}\}$, then all  cases (C\ref{C1})--(C\ref{C8})
might occur. When $\rank A_{\infty}=2$, from Lemma
\ref{lc} we have that $\Xi_{\infty}$ is a line, thus from (C\ref{C7}) we
get $|\cH \cap \cQ \cap\AG(3,q^2)|=q^3+q^2$, whereas from  (C\ref{C8})
$|\cH \cap \cQ \cap\AG(3,q^2)|=q^2$. Our Theorem \ref{main1} follows.

\section{Extremal configurations}
\label{EC}
It is possible to characterize the configurations arising when
the intersection size is either $q^2+1$ or $2q^3+q^2+1$. These are
respectively the minimum and the maximum yielded by Theorem \ref{main1}.
and they can happen only when $\cQ$ is an hyperbolic quadric.
Throughout this section we assume that the hypotheses of
Theorem \ref{main1} hold, namely that $\cH$ and $\cQ$ share a tangent
plane at some point $P$.
We prove the following two theorems.
\begin{theorem}
\label{ppt1}
Suppose $|\cH\cap\cQ|=q^2+1$. Then, $\cQ$ is a hyperbolic quadric
and $\Omega=\cH\cap\cQ$ is an elliptic quadric contained in a subgeometry
$\PG(3,q)$ embedded in $\PG(3,q^2)$.
\end{theorem}
\begin{proof}
  By Theorem \ref{main1}, $\cQ$ is hyperbolic.
  We first show that $\Omega$ must be an ovoid of $\cQ$.
  Indeed, suppose there is a generator $r$ of $\cQ$ meeting $\cH$ in
  more than $1$ point. Then, $|r\cap\cH|\geq q+1$.
  On the other hand, any generator $\ell\neq r$ of $\cQ$ belonging to the same
  regulus $\cR$ as $r$
  necessarily meets $\cH$ in at least one point. As there
  are $q^2$ such generators, we get $|\Omega|\geq q^2+q+1$ --- a
  contradiction. In particular,
  by the above argument, any generator $\ell$ of $\cQ$ through a point of
  $\Omega$ must be tangent to $\cH$. Thus, at all points $P\in\Omega$
  the tangent planes   to $\cH$ and to $\cQ$  are the same.
  A direct counting argument shows that $\Omega$ contains a $4$--simplex.
  Let $\rho$ and $\theta$ be the polarities induced respectively by
  $\cH$ and $\cQ$, and denote by $\Psi=\rho\theta$ the collineation
  they induce. By \cite[\S 83]{S}, $\rho$ and $\theta$ commute.
  Thus, $\Psi$ is an involution  fixing pointwise $\Omega$ and, with
  respect to any fixed frame, it acts as
  the conjugation $X\to X^q$.
  It follows that $\Omega$ is contained in the Baer subgeometry
  $\PG(3,q)$ fixed by $\Psi$. Actually we see that this is the complete
  intersection of $\cQ$ with this subgeometry.
\end{proof}
\begin{theorem}
\label{ppt2}
Suppose $|\cH\cap\cQ|=2q^3+q^2+1$. Then, $\cQ$ is a hyperbolic quadric
in permutable position with $\cH$. In particular, there is a quadric
$\cQ'\subseteq\cH\cap\cQ$, contained in a subgeometry $\PG(3,q)$,
such that all points of $\cH\cap\cQ$ lie on
(extended) generators of $\cQ'$.
\end{theorem}
Theorem \ref{ppt2} can be obtained as a consequence of
the analysis contained in \cite[\S 5.2.1]{E0},
in light of \cite[Lemma 19.3.1]{FPS3d}.
Here we present a direct argument.
\begin{proof}
  By Theorem \ref{main1}, $\cQ$ is a hyperbolic quadric.
  Let now $\cR_1$ and $\cR_2$ be the two reguli of $\cQ$ and
  denote by $\Omega_i$ the set of lines of $\cR_i$ which are also
  generators of $\cH$.
  Any line of
  $\cR_i$ meets $\cH$ in  either $1,q+1$ or $q^2+1$ points; let
  $r_1^i,r_2^i,r_3^i$ be the respective number of lines; in
  particular $r_3^i=|\Omega_i|$.
  Then,
  \[ r_1^i+(q+1)r_2^i+(q^2+1)(q^2+1-r_1^i-r_2^i)=2q^3+q^2+1. \]
  After some direct manipulations we get
   \[ q^2(q-1)^2=q^2(r_1^i+r_2^i)-qr_2^i; \]
   whence,
   \[ q((q-1)^2-r_1^i)=r_2^i(q-1). \]
   Thus, there are integers $t$ and $s$ such that
   $r_1^i=(q-1)s$ and $r_2^i=qt$; and
   $s+t=q-1$.
   By the above argument,
   there are at least $(q^2+1)-(q^2-q)=q+1$ generators of
   $\cQ$ in each $\cR_i$
   which belong to $\Omega_i$.
   Consider the set
   \[ \cQ'=\{ P_{\fx\fy}=\fx\cap\fy : \fx\in\Omega_1, \fy\in\Omega_2 \}. \]
   At any point $P_{\fx\fy}\in\cQ'$, the tangent plane to $\cQ$ and
   the tangent plane to $\cH$ are the same.
   Furthermore, as $q+1\geq 3$, there is at least a $4$--simplex
   contained in $\cQ'$. Thus, by \cite[\S 83]{S}, the quadric
   $\cQ$ and the Hermitian surface $\cH$ are permutable, see also
   \cite[\S 19.3]{FPS3d}, and, by
   \cite[\S 77 page 135]{S},
   $\cQ'$ is a hyperbolic quadric contained in a subgeometry
   $\PG(3,q)$.
\end{proof}


\begin{thebibliography}{999}

\bibitem{AG} A. Aguglia, V. Giordano,
  \emph{Intersections of non-classical unitals and conics in $\PG(2,q^2)$},
  Electron. J. Combin. {\bfseries 17} (2010) no. 1, Research Paper 123.
\bibitem{BB} D. Bartoli, M. De Boeck, S. Fanali, L. Storme,
  \emph{On the functional codes defined by quadrics and Hermitian varieties},
  Des. Codes Cryptogr. {\bfseries 71} (2014), no. 1, 21--46.
\bibitem{BH} A.A. Bruen, J.W.P. Hirschfeld, \emph{Intersections in projective
    space. II. Pencils of quadrics}, European J. Combin. {\bfseries 9} (1988),
  255--270.
\bibitem{CG} I. Cardinali, L. Giuzzi,
  \emph{Codes and caps from orthogonal Grassmannians},
  Finite Fields Appl. {\bfseries 24} (2013),
  148--169. 
 \bibitem{DD2} G. Donati, N. Durante, \emph{On the intersection of Hermitian
    curves and of Hermitian surfaces}, Discrete Math. {\bfseries 308} (2008),
  5196--5203.
\bibitem{DDK} G. Donati, N. Durante, G. Korchm\`aros,
  \emph{On the intersection pattern of a \\ unital and an oval in
    $\PG(2,q^2)$}, Finite Fields Appl. {\bfseries 15} (2009), 785--795.
\bibitem{DD1} N. Durante, G. Ebert, \emph{On the intersection of Hermitian
    surfaces}, Innov. Incidence Geom. {\bfseries 6/7} (2007/8), 153--167.
\bibitem{E0} F.A.B. Edoukou, \emph{Codes defined by forms of degree $2$
    on Hermitian surfaces and S\o{}renson conjecture},
  Finite Fields Appl. {\bfseries 13} (2007), 616--627.
\bibitem{EHRS10} F.A.B. Edoukou, A. Hallez, F. Rodier, L. Storme,
  \emph{A
  study of intersections of quadrics having applications on the small
  weight codewords of the functional codes $C_2(Q)$, $Q$ a
  non-singular quadric}, {J. Pure Appl. Algebra}, {\bfseries 214} (10)
(2010), 1729--1739.
\bibitem{EHRS} F.A.B. Edoukou, A. Hallez, F. Rodier, L. Storme,
\emph{ The small weight codewords of the functional codes associated to
non-singular Hermitian varieties},  {Des. Codes Cryptogr.} {\bf
56} (2010), no. 2-3, 219--233.
\bibitem{E} F.A.B. Edoukou, \emph{Codes defined by forms of
degree $2$ on non-degenerate Hermitian varieties in
$P^4(F_q)$}, \emph{Des. Codes Cryptogr.} {\bf 50} (2009), no. 1,
135--146.
\bibitem{ELX} F.A.B. Edoukou, S. Ling, C. Xing,
  \emph{Structure of functional codes defined on non--degenerate
    Hermitian varieties}, J. Combin. Theory Series A {\bfseries 118}
  (2011),  2436--2444.
\bibitem{F} W. Fulton, \emph{Intersection Theory}, Wiley, 1998.
\bibitem{G1} L. Giuzzi, \emph{On the intersection of Hermitian surfaces},
  J. Geom. {\bfseries 25} (2006),  49--60.
\bibitem{HS} A. Hallez, L. Storme,
  \emph{Functional codes arising from quadric intersections with Hermitian
    varieties},  Finite Fields Appl. {\bfseries 16} (1) (2010), 27--35.
\bibitem{FPS3d} J.W.P. Hirschfeld, {\em Finite Projective Spaces of
    Three Dimensions}, Oxford University Press, 1985.
\bibitem{HT} J.W.P. Hirschfeld, J. A.  Thas, {\em General Galois Geometries},
  Oxford University Press, 1992.
\bibitem{L} G. Lachaud,
  \emph{Number of points of plane sections and linear codes defined on \\
  algebraic varieties}
   Arithmetic, Geometry and Coding Theory,
  (Luminy, France 1993), De Gruyter (1996), 77--104.
\bibitem{S} B. Segre,
  \emph{Forme e geometrie hermitiane, con particolare riguardo al caso
    finito}, Ann. Mat. Pura Appl. (4) {\bfseries 70} (1965), 1--201.
\end{thebibliography}
\end{document}